\documentclass{amsart}

\usepackage[T1]{fontenc}
\usepackage{lmodern}
\usepackage{amsmath,amssymb,amsthm}
\usepackage[hidelinks]{hyperref}

\newcommand{\Z}{\mathbb Z}
\newcommand{\C}{\mathbb C}
\newcommand{\R}{\mathbb R}

\newcommand{\ii}{\mathrm i}
\newcommand{\floor}[1]{\left\lfloor #1\right\rfloor}
\newcommand{\conj}[1]{\overline{#1}}

\DeclareMathOperator{\rank}{rank}

\theoremstyle{plain}
\newtheorem{thm}{Theorem}[section]
\newtheorem{prop}[thm]{Proposition}

\theoremstyle{definition}

\theoremstyle{remark}

\title{The rank of the modular multiplication matrix}

\author{Haojia Shi}
\address{School of Mathematical Sciences, Peking University}
\email{2500010808@stu.pku.edu.cn}

\date{}
\hypersetup{
  pdftitle={The rank of the modular multiplication matrix},
  pdfauthor={Haojia Shi},
  pdfsubject={Matrix rank and modular multiplication matrices},
  pdfkeywords={modular multiplication matrix, matrix rank, finite Fourier transform, Dirichlet characters}
}

\begin{document}

\begin{abstract}
We prove Bueno et al.'s conjectural formula for the rank of the modular multiplication matrix for every positive integer. The proof first applies the additive Fourier transform to split the row space by exact additive character order, and then applies the multiplicative character decomposition on unit groups to count the nonzero components. The same formula holds after deleting the zero row and zero column.
\end{abstract}

\maketitle

\section{Introduction}\label{sec:introduction}

The modular multiplication matrix records the least nonnegative residues of products modulo an integer. Bueno et al. studied the kernel of this matrix in the prime case and stated the general rank formula as Conjecture~16 in the published version of \cite{Bueno_Furtado_Karkoska_Mayfield_Samalis_Telatovich_2016}; the same statement appears as Conjecture~14 in the 2015 preprint. We prove this conjecture for all positive integers.

For a positive integer $n$, let $C_n$ be the real $n$ by $n$ matrix with
\[
        C_n(i,j)\equiv ij \pmod n,\qquad 0\le C_n(i,j)\le n-1,
\]
where $1\le i,j\le n$. Let $H_n$ be the matrix obtained from $C_n$ by deleting the last row and the last column. Let $k(n)$ denote the number of proper positive divisors of $n$.

\begin{thm}[Rank of the modular multiplication matrix]\label{thm:main}
For every positive integer $n$,
\[
        \rank_{\R}(C_n)=\rank_{\R}(H_n)=\floor{\frac{n-1}{2}}+k(n).
\]
\end{thm}

The proof has two transforms. Section~\ref{sec:additive} computes the additive Fourier transform of each nonzero row and separates the row space by exact additive character order. Section~\ref{sec:characters} computes which multiplicative character components survive in each unit-group layer. Section~\ref{sec:rank} combines the layers and proves Theorem~\ref{thm:main}.

\section{Additive Fourier transform}\label{sec:additive}

In this section, we compute the additive Fourier transform of each nonzero row of $C_n$.

\begin{prop}[Additive Fourier formula for the rows]\label{prop:additive}
Let $n$ be a positive integer. Index the rows and columns of $C_n$ by the residue classes in $\Z/n\Z$, and let $C_n(a,b)$ be the least nonnegative residue of $ab$ modulo $n$. For every divisor $m$ of $n$ with $m>1$, put $d=n/m$ and $U_m=(\Z/m\Z)^\times$. For $u\in U_m$, let $R_{m,u}$ be the row vector
\[
        R_{m,u}(b)=C_n(du,b).
\]
Let $\zeta_m=\exp(2\pi \ii/m)$, and for a row vector $R$ define its unnormalized additive Fourier coefficient by
\[
        \widehat R(t)=\sum_{b\in \Z/n\Z}R(b)\exp(-2\pi \ii tb/n),
        \qquad t\in \Z/n\Z.
\]
Then the nonzero rows of $C_n$ are exactly the vectors $R_{m,u}$, with $m>1$ dividing $n$ and $u\in U_m$. Moreover $\widehat R_{m,u}(t)=0$ unless $t=ds$ for some $s\in \Z/m\Z$. If $t=ds$, then
\[
        \widehat R_{m,u}(t)=d^2m(m-1)/2
\]
for $s=0$, and
\[
        \widehat R_{m,u}(t)=
        -\frac{d^2m}{1-\zeta_m^{-su^{-1}}}
\]
for $s\ne 0$ in $\Z/m\Z$, where $u^{-1}$ is the inverse of $u$ modulo $m$.
\end{prop}

\begin{proof}
Let $d=n/m$. If $a$ is a nonzero residue class modulo $n$, set $d_a=\gcd(a,n)$, $m_a=n/d_a$, and let $u_a$ be the class of $a/d_a$ modulo $m_a$. Then $u_a$ is a unit modulo $m_a$ and $a=d_au_a$ in $\Z/n\Z$. This gives a unique pair $(m_a,u_a)$ with $m_a>1$ and $u_a\in U_{m_a}$. Conversely every pair $(m,u)$ with $m>1$ and $u\in U_m$ gives the nonzero row index $du$ modulo $n$. Thus the nonzero rows are exactly the vectors $R_{m,u}$.

Fix $m$ and $u$, and put $d=n/m$. For $b\in \Z/n\Z$, write $x$ for the image of $b$ in $\Z/m\Z$. Since $dub$ is divisible by $d$ modulo $n$, the least nonnegative residue of $dub$ modulo $n$ is $d$ times the least nonnegative residue of $ux$ modulo $m$. Hence
\[
        R_{m,u}(b)=d\,r_m(ux),
\]
where $r_m(y)$ denotes the representative of $y$ in $\{0,\ldots,m-1\}$.

For $t\in \Z/n\Z$, write $b=x+km$ with $x\in\{0,\ldots,m-1\}$ and $k\in\{0,\ldots,d-1\}$. Then
\[
\widehat R_{m,u}(t)
=
d\sum_{x=0}^{m-1} r_m(ux)\exp(-2\pi \ii tx/n)
 \sum_{k=0}^{d-1}\exp(-2\pi \ii tkm/n).
\]
The inner geometric sum is
\[
        \sum_{k=0}^{d-1}\exp(-2\pi \ii tk/d).
\]
It is $0$ unless $t$ is divisible by $d$ in $\Z/n\Z$, and it is $d$ when $t=ds$ for a unique $s\in \Z/m\Z$. Therefore $\widehat R_{m,u}(t)=0$ unless $t=ds$. In the remaining case,
\[
        \widehat R_{m,u}(ds)
        =
        d^2\sum_{x=0}^{m-1}r_m(ux)\zeta_m^{-sx}.
\]
If $s=0$, multiplication by $u$ permutes $\Z/m\Z$, so the sum is
\[
        0+1+\cdots+(m-1)=m(m-1)/2.
\]
If $s$ is nonzero, substitute $y=ux$. Then
\[
        \sum_{x=0}^{m-1}r_m(ux)\zeta_m^{-sx}
        =
        \sum_{y=0}^{m-1}y(\zeta_m^{-su^{-1}})^y.
\]
The number $q=\zeta_m^{-su^{-1}}$ satisfies $q^m=1$ and $q\ne 1$. Differentiating
\[
        \sum_{y=0}^{m-1}q^y=\frac{1-q^m}{1-q}
\]
and multiplying by $q$ gives
\[
        \sum_{y=0}^{m-1}yq^y=-\frac{m}{1-q}.
\]
This gives the asserted formula.
\end{proof}

\section{Character decomposition on unit groups}\label{sec:characters}

The next proposition determines the multiplicative character coefficients of the finite residue Fourier kernel.

\begin{prop}[Multiplicative character coefficients]\label{prop:characters}
Let $q$ be an integer with $q\ge 2$, let $U_q=(\Z/q\Z)^\times$, let $\zeta_q=\exp(2\pi \ii/q)$, and let $\chi$ be a complex multiplicative character of $U_q$. Define
\[
        A(q,\chi)=\sum_{a\in U_q}\frac{\chi(a)}{1-\zeta_q^a}.
\]
Then $A(q,\chi)$ is nonzero if $\chi$ is the principal character, $A(q,\chi)=0$ if $\chi$ is nonprincipal and $\chi(-1)=1$, and $A(q,\chi)$ is nonzero if $\chi(-1)=-1$.
\end{prop}

\begin{proof}
Suppose first that $\chi(-1)=1$. The identity
\[
        \frac{1}{1-z}+\frac{1}{1-z^{-1}}=1
\]
holds for $z\ne 1$. If $q=2$, the group $U_q$ has one element and its only character is principal, so $A(q,\chi)=1/2$. If $q>2$, no unit $a$ satisfies $a=-a$ in $\Z/q\Z$, so the units split into pairs $\{a,-a\}$. Since $\zeta_q^a\ne 1$ for $a\in U_q$, pairing $a$ with $-a$ gives
\[
        A(q,\chi)=\frac{1}{2}\sum_{a\in U_q}\chi(a).
\]
If $\chi$ is principal, this is $\varphi(q)/2$ and is nonzero. If $\chi$ is nonprincipal, choose $g\in U_q$ with $\chi(g)\ne 1$. Multiplication by $g$ permutes $U_q$, so
\[
        \sum_{a\in U_q}\chi(a)
        =
        \sum_{a\in U_q}\chi(ga)
        =
        \chi(g)\sum_{a\in U_q}\chi(a),
\]
and hence the sum is $0$. This proves the principal and nonprincipal even cases.

Now assume $\chi(-1)=-1$. Let $f$ be the conductor of $\chi$, and let $\psi$ be the primitive Dirichlet character modulo $f$ that induces $\chi$. Then $\psi(-1)=-1$. We reduce $A(q,\chi)$ to the primitive modulus $f$.

For every multiple $Q$ of $f$, define
\[
        S(Q)=\sum_{b=1}^{Q-1}
        \frac{\psi(b)}{1-\exp(2\pi \ii b/Q)},
\]
where $\psi(b)$ is evaluated modulo $f$ and is $0$ when $b$ is not coprime to $f$. Write $Q=fR$. For each $c\in\{1,\ldots,f-1\}$ with $\gcd(c,f)=1$, the terms $b=c+fj$, $0\le j\le R-1$, contribute
\[
        \psi(c)\sum_{j=0}^{R-1}\frac{1}{1-\alpha\eta^j},
\]
where $\alpha=\exp(2\pi \ii c/Q)$ and $\eta=\exp(2\pi \ii/R)$. Since
\[
        \alpha^R=\exp(2\pi \ii c/f)\ne 1,
\]
the root-of-unity identity
\[
        \sum_{j=0}^{R-1}\frac{1}{1-\alpha\eta^j}
        =
        \frac{R}{1-\alpha^R}
\]
applies. Summing over $c$ gives
\[
        S(Q)=\frac{Q}{f}A(f,\psi).
\]

The sum $A(q,\chi)$ is obtained from $S(q)$ by also imposing coprimality to the primes $p$ that divide $q$ but do not divide $f$. Let $P$ be this finite set of primes. Inclusion-exclusion gives
\[
        A(q,\chi)
        =
        \sum_D\mu(D)
        \sum_{\substack{1\le a\le q-1\\ D\mid a}}
        \frac{\psi(a)}{1-\zeta_q^a},
\]
where $D$ ranges over squarefree products of primes in $P$. For such $D$, the integer $D$ is coprime to $f$. Writing $a=Db$ transforms the inner sum into $\psi(D)S(q/D)$. Therefore
\[
        A(q,\chi)
        =
        \frac{q}{f}A(f,\psi)\sum_D\frac{\mu(D)\psi(D)}{D}
        =
        \frac{q}{f}A(f,\psi)\prod_{p\in P}\left(1-\frac{\psi(p)}{p}\right).
\]
Each factor is nonzero because $|\psi(p)|\le 1$ and $p\ge 2$. Thus $A(q,\chi)$ is nonzero once $A(f,\psi)$ is nonzero.

It remains to prove $A(f,\psi)$ is nonzero. The finite identity
\[
        \frac{1}{1-z}=-\frac{1}{f}\sum_{r=1}^{f-1}rz^r
\]
is valid whenever $z^f=1$ and $z\ne 1$. Applying it with $z=\zeta_f^a$ gives
\[
        A(f,\psi)
        =
        -\frac{1}{f}
        \sum_{r=1}^{f-1}r\sum_{a=1}^{f}\psi(a)\zeta_f^{ar}.
\]
For a primitive Dirichlet character, the inner Gauss sum equals $\conj{\psi}(r)\tau(\psi)$, where
\[
        \tau(\psi)=\sum_{a=1}^{f}\psi(a)\zeta_f^a
\]
and $\conj{\psi}(r)=0$ when $r$ is not coprime to $f$. Hence
\[
        A(f,\psi)
        =
        -\frac{\tau(\psi)}{f}
        \sum_{r=1}^{f} r\,\conj{\psi}(r).
\]
For an odd primitive Dirichlet character $\psi$ of conductor $f$, the Washington formula recorded in \cite[Theorem~1.1]{HKL19_Theorem_1_1} gives
\[
        L(1,\psi)=
        \frac{\pi \ii\,\tau(\psi)}{f^2}
        \sum_{r=1}^{f}\conj{\psi}(r)r.
\]
We use the convention $L(1,\psi)$ for Dirichlet $L$-functions; Kahn writes the same value as $L(\psi,1)$. Dirichlet's nonvanishing theorem, in Kahn's convention \cite[Theorem~4.2.11]{Kahn15_Theorem_4_2_11}, says that $L(\psi,1)\ne 0$ for every nontrivial Dirichlet character $\psi$. Since an odd primitive character is nontrivial, $L(1,\psi)$ is nonzero in our convention. The displayed formula then implies that
\[
        \tau(\psi)\sum_{r=1}^{f}r\,\conj{\psi}(r)
\]
is nonzero. Therefore $A(f,\psi)$ is nonzero, and the conductor reduction above gives $A(q,\chi)\ne 0$.
\end{proof}

\section{Rank count}\label{sec:rank}

We now combine the additive exact-order decomposition with Proposition~\ref{prop:characters}.

\begin{proof}[Proof of Theorem~\ref{thm:main}]
For $n=1$, the matrix $C_1$ is the $1$ by $1$ zero matrix, $H_1$ is the empty $0$ by $0$ matrix, $k(1)=0$, and the formula gives $0$. Assume $n\ge 2$.

Permute rows and columns so that they are indexed by $\Z/n\Z$, with the original index $n$ corresponding to $0$. This does not change rank. The row and column indexed by $0$ are zero, so deleting them does not change rank. Hence
\[
        \rank_{\R}(C_n)=\rank_{\R}(H_n).
\]
Since $C_n$ has real entries, its rank over $\R$ equals the rank over $\C$ of the same matrix. We compute the complex row rank after the invertible additive Fourier transform on the column variable $b\in \Z/n\Z$.

For every divisor $q$ of $n$, let $E_q$ be the coordinate subspace in Fourier variables supported on the additive characters of exact order $q$. Thus $E_1$ is the zero-frequency line, and for $q>1$ the coordinates of $E_q$ are indexed by
\[
        X_q=\{(n/q)s\in \Z/n\Z\mid s\in U_q\},
        \qquad U_q=(\Z/q\Z)^\times.
\]
The nonzero Fourier coordinates are the direct sum of $E_q$ over divisors $q>1$ of $n$.

By Proposition~\ref{prop:additive}, every nonzero row is $R_{m,u}$ for a unique divisor $m>1$ of $n$ and $u\in U_m$, and its Fourier transform is supported only on $E_q$ with $q$ dividing $m$. More explicitly, if $q$ divides $m$ and $t=(n/q)s$ with $s\in U_q$, then the $E_q$-coordinate of the Fourier transform of $R_{m,u}$ is a nonzero scalar, depending on $n,m,q$ but not on $s$, times
\[
        h_q(s\,\overline u^{-1}),
        \qquad
        h_q(x)=\frac{1}{1-\zeta_q^{-x}},
\]
where $\overline u$ is the image of $u$ in $U_q$. If $q$ does not divide $m$, the $E_q$ projection is zero. This follows from Proposition~\ref{prop:additive} by writing the frequency $t$ also as
\[
        t=(n/m)((m/q)s),
\]
so that
\[
        \zeta_m^{-(m/q)su^{-1}}=\zeta_q^{-s\overline u^{-1}}.
\]

We next compute the dimension of the span of the multiplicative translates of $h_q$ on $U_q$. The multiplicative characters of the finite abelian group $U_q$ form a basis for all complex-valued functions on $U_q$. For a function $h$ on $U_q$, the span of the right translates
\[
        x\mapsto h(xv^{-1}),\qquad v\in U_q,
\]
has dimension equal to the number of multiplicative characters whose Fourier coefficient in $h$ is nonzero: writing $h$ as a sum of character components, translation multiplies each character component by a scalar depending on $v$, and character orthogonality isolates exactly the nonzero components. For
\[
        h_q(x)=\frac{1}{1-\zeta_q^{-x}},
\]
the coefficient of a character $\chi$ is, up to the nonzero scalar $\chi(-1)$, the sum $A(q,\conj{\chi})$ from Proposition~\ref{prop:characters}. Therefore Proposition~\ref{prop:characters} implies that the translate span of $h_q$ has dimension
\[
        b(q)=1+o(q),
\]
where $o(q)$ is the number of characters $\chi$ of $U_q$ with $\chi(-1)=-1$. The one additional character is the principal character; the nonprincipal even characters contribute zero.

This gives the upper bound. Every row has entry $0$ in the column $b=0$. Therefore any linear combination of rows that has only zero additive Fourier frequency is a constant vector whose value at $b=0$ is $0$, hence is the zero vector. Thus projection of the row space to the direct sum of the nonzero exact-order spaces $E_q$, $q>1$, is injective. For each fixed $q>1$, the preceding paragraph shows that the projection of all rows to $E_q$ has dimension at most $b(q)$. Hence
\[
        \rank_{\C}(C_n)\le \sum_{\substack{q\mid n\\ q>1}} b(q).
\]

For the lower bound, list the divisors $m>1$ of $n$ in increasing numerical order and let $W_j$ be the span of the rows $R_{m,u}$ for the first $j$ listed divisors $m$ and all $u\in U_m$. When the divisor $m$ is first added, all previously listed divisors are smaller than $m$ and therefore none is divisible by $m$. By Proposition~\ref{prop:additive}, the previous rows have zero projection to $E_m$. The rows $R_{m,u}$ with $u\in U_m$ have $E_m$ projections equal to nonzero scalar multiples of the translates of $h_m$, and those translates span a space of dimension $b(m)$ by the preceding character calculation. Therefore adding the layer $m$ increases the dimension by at least $b(m)$. Induction over the ordered divisors gives
\[
        \rank_{\C}(C_n)\ge \sum_{\substack{m\mid n\\ m>1}} b(m).
\]
The upper and lower bounds agree, so
\[
        \rank_{\C}(C_n)=\sum_{\substack{q\mid n\\ q>1}} b(q).
\]

It remains to evaluate this divisor sum. For $q=2$, the element $-1$ equals $1$ in $U_2$, so $o(2)=0$. For $q>2$, the element $-1$ is a nonidentity element of $U_q$. Evaluation at $-1$ is then a nontrivial homomorphism from the character group of $U_q$ to $\{1,-1\}$, so exactly half of the $\varphi(q)$ characters are odd. Hence
\[
        o(q)=\varphi(q)/2
        \qquad (q>2).
\]
Consequently
\[
        \sum_{\substack{q\mid n\\ q>1}} b(q)
        =
        \#\{q:q\mid n,\ q>1\}
        +
        \sum_{\substack{q\mid n\\ q>1}}o(q).
\]
The first term is $k(n)$, because $q>1$ corresponds to the proper divisor $n/q$ of $n$. For the second term, use the identity
\[
        \sum_{q\mid n}\varphi(q)=n.
\]
If $n$ is odd, the divisors $q>1$ are all greater than $2$, so
\[
        \sum_{\substack{q\mid n\\ q>1}}o(q)
        =
        \frac{1}{2}\sum_{\substack{q\mid n\\ q>1}}\varphi(q)
        =
        \frac{n-1}{2}.
\]
If $n$ is even, the $q=2$ divisor contributes no odd characters, so
\[
        \sum_{\substack{q\mid n\\ q>1}}o(q)
        =
        \frac{1}{2}\sum_{\substack{q\mid n\\ q>2}}\varphi(q)
        =
        \frac{n-2}{2}.
\]
In both cases this is $\floor{(n-1)/2}$. Therefore
\[
        \rank_{\R}(C_n)=\rank_{\R}(H_n)=\floor{\frac{n-1}{2}}+k(n),
\]
as claimed.
\end{proof}

\section*{Acknowledgements}
The author used Danus for drafting and editing assistance. The author checked the mathematical arguments, citations, and final text.

\end{document}